\documentclass{amsart}

\usepackage{hyperref}
\usepackage{hhline} 
\usepackage{array}              
\usepackage{amssymb}              
\usepackage{lmodern,bm}              
\usepackage{longtable}
\usepackage{tikz}
\usepackage[all]{xy}
\usepackage{dsfont}
\usepackage{pdflscape}
\usepackage{bigdelim}
\usepackage{multirow}
\usepackage{wasysym}
\usepackage{enumitem}

\CompileMatrices

\newtheorem{theorem}{Theorem}[section]
\newtheorem{introthm}{Theorem}
\newtheorem{introlem}{Lemma}
\newtheorem{introcor}{Corollary}
\newtheorem{lemma}[theorem]{Lemma}
\newtheorem{question}{Question}
\newtheorem{proposition}[theorem]{Proposition}
\newtheorem{corollary}[theorem]{Corollary}

\theoremstyle{definition}

\newtheorem{definition}[theorem]{Definition}

\newtheorem{example}[theorem]{Example}
\newtheorem{remark}[theorem]{Remark}

\numberwithin{equation}{theorem}

\newlength{\dovheight}
\newcommand{\doubleoverline}[1]{%
    \settoheight{\dovheight}{\ensuremath{\overline{#1}}}%
    \addtolength{\dovheight}{-0.1ex}%
    \overline{\vphantom{\rule{1pt}{\dovheight}}%
    \smash{\overline{#1}}}}

\newlength{\dhatheight}
\newcommand{\doublewidehat}[1]{%
    \settoheight{\dhatheight}{\ensuremath{\widehat{#1}}}%
    \addtolength{\dhatheight}{-0.35ex}%
    \widehat{\vphantom{\rule{1pt}{\dhatheight}}%
    \smash{\widehat{#1}}}}
    
    \newlength{\dhatheightX}
\newcommand{\doublewidehatX}[1]{%
    \settoheight{\dhatheightX}{\ensuremath{\widehat{#1}}}%
    \addtolength{\dhatheightX}{-0.35ex}%
    \widehat{\vphantom{\rule{1pt}{\dhatheightX}}%
    \smash{\widehat{\!#1}}}}

\newlength{\dhatheightt}
\newlength{\Ywidth}
\newcommand{\git}{\mathbin{
  \mathchoice{/\mkern-6mu/}
    {/\mkern-6mu/}
    {/\mkern-5mu/}
    {/\mkern-5mu/}}}

\def\vector2#1#2{\left(\begin{array}{c} #1 \\ #2 \end{array}\right)}

\def\klt{{\rm klt}}

\def\Cox{\mathcal{R}}

\def\MDS{{\rm MDS}}
\def\Sing{{\rm Sing}}
\def\CR{{\rm CR}}

\def\Pic{{\rm Pic}}
\def\Cl{{\rm Cl}}

\def\Xendl{\mathfrak{X}}

\def\CC{{\mathbb C}}
\def\KK{{\mathbb K}}
\def\TT{{\mathbb T}}
\def\ZZ{{\mathbb Z}}

\def\NN{{\mathbb N}}
\def\QQ{{\mathbb Q}}

\def\reg{ \mathrm{reg}}

\def\Chi{{\mathbb X}}

\def\GIT{{\rm GIT}}

\def\Pic{{\rm Pic}}

\def\Spec{{\rm Spec}}

\def\reg{{\rm reg}}

\title[Gorensteinness and iteration of Cox rings]{Gorensteinness and iteration of Cox rings \\ for Fano type varieties}
\author[L.~Braun]{Lukas Braun}

\address{Mathematisches Institut, Universit\"at T\"ubingen,
Auf der Morgenstelle 10, 72076 T\"ubingen, Germany}
\email{braun@math.uni-tuebingen.de}


\begin{document}

\begin{abstract}
We show that finitely generated Cox rings are Gorenstein. This leads to a refined characterization of varieties of Fano type: they are exactly those projective varieties with Gorenstein canonical quasicone Cox ring.
We then show that for varieties of Fano type and Kawamata log terminal quasicones, iteration of Cox rings is finite with factorial master Cox ring. Moreover, we prove a relative version of Cox ring iteration for \emph{almost principal} solvable $G$-bundles.
\end{abstract}

\maketitle

\section*{Introduction}
Let $X$ be a normal algebraic variety over the field $\CC$ of complex numbers. If $X$ has finitely generated divisor class group $\Cl(X)$, one can define its $\Cl(X)$-graded \emph{Cox ring}
$$
\Cox(X):=\bigoplus_{\Cl(X)} \Gamma(X,\mathcal{O}_X(D))
$$
and ask if it is a finitely generated $\CC$-algebra. If so, $X$ is called a \emph{Mori Dream Space} ($\MDS$) due to Hu and Keel~\cite{HuKeel}, who showed that this property is equivalent to a good behaviour with respect to the minimal model program, whence the name. 

Note that this definition differs from the original one of Hu and Keel, as we do not assume $X$ to be projective and $\QQ$-factorial, and in that $\Cox(X)$ is graded by $\Cl(X)$, not $\Pic(X)$. Moreover, some care has to be taken if $\Cl(X)$ has torsion.
The original reference for this case is~\cite{BerchHaus}, where also the case of non-finitely generated $\Cox(X)$ is investigated. See also the comprehensive work~\cite{coxrings}. 

Examples of $\MDS$ include toric~\cite{cox} and more general spherical~\cite{brion} varieties as well as certain ones with an action of a torus of lower dimension, see~\cite{hausüß}. It was  shown in~\cite{birketal} that varieties of \emph{Fano type} - i.e. projective varieties admitting an effective $\QQ$-divisor $\Delta$ such that  $-(K_X+\Delta)$ is ample and $(X,\Delta)$ is \emph{Kawamata log terminal} ($\klt$) - are $\MDS$. 

Cox rings of $\MDS$ enjoy many nice properties: 
The $\Cl(X)$-grading induces an action of the characteristic quasitorus $H_X:=\Spec\, \CC[\Cl(X)]$ on the \emph{total coordinate space} $\overline{X}:=\Spec\, \Cox(X)$, and there is an open subvariety $\widehat{X}\subseteq \overline{X}$ with complement of codimension at least two such that $X$ is isomorphic to the good quotient $\widehat{X}\git H_X$. We call $\widehat{X}$ the \emph{characteristic space}.

In~\cite{BerchHaus} and with different methods in~\cite{fac} and~\cite{arzhfac} it was shown that if $\Cl(X)$ is free, then $\Cox(X)$ is \emph{factorial} and thus in particular it is Gorenstein, i.e.   the canonical class $K_{\overline{X}}$ is Cartier. If  $\Cl(X)$ is not free, then $\Cox(X)$ is still \emph{factorially graded} in the sense that homogeneous elements can be expressed as a unique product of irreducible homogeneous ones, see~\cite{arzhfac}, but properties such as  Gorensteinness do not follow from this. Moreover, it was shown independently in~\cite{gongyo, brown, kaw} that a projective variety $X$ is of Fano type if and only if $\Cox(X)$ is log terminal. Analogously, (projective) $\MDS$ are shown to be of Calabi-Yau type if and only if $\Cox(X)$ is log canonical~\cite{kaw, gongyo}. In~\cite{brown}, Brown also observed that if $X$ is smooth and Fano, then  $\Cox(X)$ is Gorenstein canonical. Our first result is the following:

\begin{introthm}
\label{thm:Gor}
Let $X$ be an $\MDS$. Then $\Cox(X)$ is Gorenstein.
\end{introthm}

As Brown points out~\cite{brown}, a finitely generated Cox ring need not be Cohen Macaulay. But since $K_{\overline{X}}$ is Cartier, if e.g. $\Cox(X)$ has rational singularities, it is not only Cohen Macaulay, but even canonical. 

Recall that a \emph{quasicone} is an affine variety with a $\CC^*$-action such that all closures of $\CC^*$-orbits meet in one common point, the \emph{vertex}. We will call the coordinate ring of such affine variety a quasicone as well. In fact, for an $\MDS$ that is either complete or itself  a quasicone, $\Cox(X)$ is a quasicone, cf. Proposition~\ref{prop:quasicone2}.

By~\cite[Le.~5.1]{murthy}, quasicones have trivial Picard group  (normality is essential here, see Remark~\ref{rem:PicNorm}), yet the whole class group is generated by the local class group at the vertex and $K_X$ is even trivial if $X$ is Gorenstein. 
In fact, $\klt$ quasicones turn out to be the affine counterparts of Fano type varieties, so we can refine the characterization from~\cite{gongyo, brown, kaw}:

\begin{introthm}
\label{thm:charsing}
Let $X$ be projective (affine). Then it is of Fano type (a $\klt$ quasicone) if and only if $\Cox(X)$ is a Gorenstein canonical quasicone.
\end{introthm}

Note that the property 'Fano type / $\klt$ quasicone' does not only guarantee $\MDS$-ness, but is also \emph{preserved} when passing to the total coordinate space $\overline{X}$. Thus it is natural to consider the Cox ring and total coordinate space $X^{(2)}:=\doubleoverline{X}$ of $X^{(1)}:=\overline{X}$, then the ones of $X^{(2)}$ etc., we call this the \emph{iteration of Cox rings}. 

Iteration of Cox rings has been introduced by Arzhantsev, Hausen, Wrobel and the author in~\cite{ltpticr} for affine rational varieties with a torus action of complexity one.
In general, for an $\MDS$ $X$, three scenarios are possible: either $X$ has \emph{infinite} iteration of Cox rings, or at some point $X^{(N)}$, the iteration stops. The latter may have two reasons: either $X^{(N)}$ is non-$\MDS$ or it is factorial, i.e. its own total coordinate space. We call $\CC[X^{(N)}]$ the \emph{master Cox ring} of $X$.

It was shown in~\cite{ltpticr} that log terminal quasicones with a torus action of complexity one have finite iteration of Cox rings with factorial master Cox ring. In~\cite{hauswrob}, Hausen and Wrobel determined which non-quasicone affine rational varieties with a torus action of complexity one have finite iteration with factorial master Cox ring. In fact, it can be deduced from their observations that \emph{all} rational varieties with a torus action of complexity one have \emph{finite} iteration of Cox rings, though with possibly non-$\MDS$ master Cox ring. 
Toric varieties - in a trivial way, since $\Cox(X)$ is  a polynomial ring - and also spherical varieties - see~\cite[Thm.~1.1]{gagliardi} - have factorial master Cox ring. We prove the following:

\begin{introthm}
\label{thm:CRI}
Let $X$ be of Fano type or a $\klt$ quasicone. Then $X$ has finite iteration of Cox rings with factorial master Cox ring.
\end{introthm}

Note that for any $i\leq N$, the quasitorus quotients in the iteration of Cox rings can be combined to express $X$ as a $\GIT$-quotient ${X^{(i)}}^{\mathrm{ss}}\git G$ of $X^{(i)}$ by a solvable group $G$, see Proposition~\ref{prop:facquot}. In the case of log terminal surface singularities, this yields the presentation as finite quotients of $\CC^2$ and, in addition, the derived normal series of the respective finite groups, see~\cite{ltpticr}.
Note that our observations till now guarantee $\MDS$-ness of all $X^{(i)}$ in the iteration of Cox rings for $\klt$ quasicones and thus Fano type varieties, but they \emph{do not yield finiteness} of the iteration. To prove this, we use a theorem of Greb, Kebekus and Peternell~\cite[Thm.~1.1]{grebkebpet}, stating (roughly) that in a sequence of finite morphisms between $\klt$ spaces that are \'etale over smooth loci all but finitely many must be \'etale. 


In order to apply this result, we need to pass from in general nonfinite quasitorus quotients to finite Galois covers. This is made possible by the observation made in Corollary~\ref{cor:TandE}, which we state here as Lemma~\ref{le:TandE} as it is essential for most of our statements, including Theorem~\ref{thm:Gor} on Gorensteinness of the Cox ring. For convenience, we will mostly depict torus quotients by vertical, finite quotients by horizontal and quasitorus quotients by diagonal arrows.

\begin{introlem}\label{le:TandE}
Let $X$ be an $\MDS$ and $H_X= E \times \TT$ its characteristic quasitorus with torsion and torus part $E$ and $\TT$ respectively. Then:
\begin{enumerate}
\item The Cox rings of $X$ and of the geometric quotient $X_E:=\widehat X / E$ coincide. In particular, $X_E$ is a quasiaffine $\MDS$ and  almost factorial.
\item The characteristic space $\widehat{X}$ is an $\MDS$ if and only if $X_\TT:=\widehat X \git \TT$ is an $\MDS$ and in that case, their Cox rings coincide.
\end{enumerate}
In particular, we have the following commutative diagram of quasitorus quotients, where $\CR$ denotes characteristic spaces:
$$
 
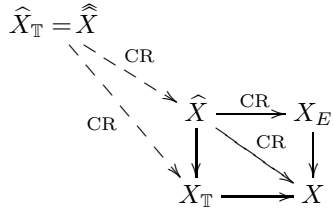
\xymatrix@R=1.5em{ 
{\widehat{X}_\TT=\doublewidehatX{X}}
\ar@{-->}[rd]^{\CR}
\ar@{-->}[rdd]_{\CR}
\\
& {\widehat{X}}
\ar[r]^{\CR}
\ar[d]
\ar[rd]^{\CR}
&
X_E
\ar[d]
\\
&
X_\TT
\ar[r]
&
X
}
$$
\end{introlem}
Here, (i)  is essential in the proof of Theorem~\ref{thm:Gor} as it allows to reduce to the case of finite $\Cl(X)$ while (ii) allows to pass from quasitorus quotients to finite Galois covers in the proof of Theorem~\ref{thm:CRI}.
The first main ingredience of Lemma~\ref{le:TandE} is a theorem of Bechtold~\cite[Thm.~1.5]{Bechtold}, stating that \emph{factoriality of a grading} (which is essential for Cox rings) is preserved when changing the \emph{free} part of the grading group.  
The second underlying observation is that if an $\MDS$ $X$ has a \emph{quotient presentation}  $X=Y\git H$ with a variety $Y$ and a quasitorus $H$ with certain good properties, see~\cite[Def.~4.2.1.1]{coxrings}, then the presentation $\widehat{X}\to X$ factors through $Y \to X$, see~\cite[Thm.~4.2.1.4]{coxrings}. In analogy to~\cite{coxrings}, the term quotient presentation will be used only for abelian quotients in the following.

In view of this, consider Okawa's  Theorem~\cite[Thm.~1.1]{okawa}, which states that for a surjective morphism $f \colon X \to Y$ between projective varieties, if $X$ is an $\MDS$ then $Y$ is as well. We can in fact say more if $f$ is a (quasitorus) quotient presentation, namely if $Y$ and $\widehat{Y}$ are $\MDS$, then $X$ is so, see Proposition~\ref{prop:relMDS}. In terms of iteration of Cox rings, this generalizes to the following:

\begin{introthm}
\label{thm:webCRI}
Let 
$
~
\cdots\!\to \! X_4 \! \to \! X_3 \! \to \! X_2 \! \to \! X_1
$ be a chain of quotient presentations.
Denote the $i$-th iterated characteristic space of $X_j$ by $X_j^{(i)}$ if it exists, i.e. $X_j^{(1)}:=\widehat{X}_j$ etc. Then if one of the $X_i$ has infinite iteration of Cox rings, the others have as well. If one has factorial master Cox ring, the others have as well and in that case, all master Cox rings coincide. In these cases, we get a \emph{web of Cox ring iterations}:
$$
 
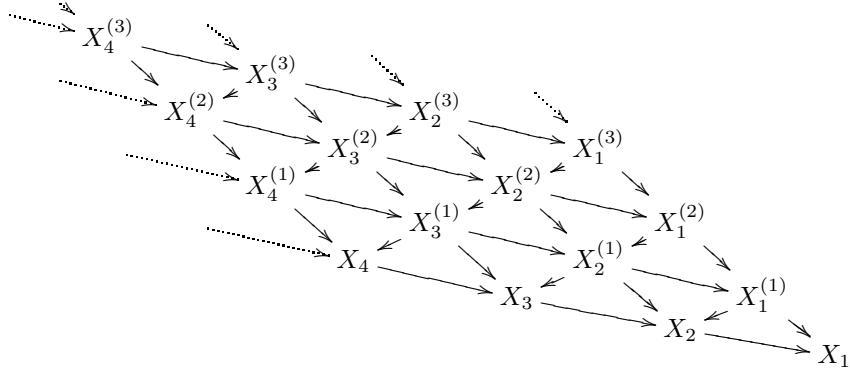
\xymatrix@C=0.5em@R=-0.5em@H=\dhatheightt@W=\Ywidth{ 
&\ar@{..>}[rdd] 
 \\
 \ar@{..>}[rrd]&&&\ar@{..>}[rdd]
 \\&&
 X_4^{(3)}
 \ar[rrd]
\ar[rdd]
&&&\ar@{..>}[rdd]
 \\
  & \ar@{..>}[rrd]&&&
   X_3^{(3)}
\ar[ld]
 \ar[rrd]
\ar[rdd]&&&
\ar@{..>}[rdd]
 \\
&&&X_4^{(2)}
 \ar[rrd]
\ar[rdd]
&&& X_2^{(3)}
\ar[ld]
 \ar[rrd]
\ar[rdd]
\\
&&\ar@{..>}[rrd]&&&X_3^{(2)}
\ar[ld]
 \ar[rrd]
\ar[rdd]&&  &
X_1^{(3)}
 \ar[ld]
\ar[rdd]
 \\
&&&&X_4^{(1)}
\ar[rrd]
\ar[rdd]&&& 
X_2^{(2)}
\ar[ld]
 \ar[rrd]
\ar[rdd]
\\
&&&\ar@{..>}[rrd]&&&X_3^{(1)}
\ar[ld]
\ar[rrd]
\ar[rdd]&& &
X_1^{(2)}
 \ar[ld]
\ar[rdd]
 \\
&&&&& X_4
\ar[rrd]
&&& X_2^{(1)}
\ar[ld]
\ar[rrd]
\ar[rdd]
\\
&&&&&&&
X_3
\ar[rrd]
&& & 
X_1^{(1)}
\ar[ld]
\ar[rdd]
\\
&&&&&&&&& 
X_2
\ar[rrd]
\\
&&&&&&&&&&&
X_1
}
$$
\end{introthm}

Note that if for example $X_1$ has finite iteration of Cox rings with non-$\MDS$ master Cox ring $\CC[X_1^{(3)}]$, we get $\MDS$-ness only for $X_2, X_2^{(1)}, X_3$. So in such case, the maximal number of steps in the Cox ring iteration becomes important. 

Due to~\cite[Thm. 4.2.1.4]{coxrings}, the characteristic space $\widehat{X} \to X$ is a universal reductive abelian (possibly nondiscrete) 'cover' of $X$ in the sense that it factors through any quotient presentation $Y \to X$. This can be generalized for $\MDS$ with finite iteration of Cox rings and factorial master Cox ring in the following way.

\begin{introcor}
\label{cor:unisolv}
Let $X$ be an $\MDS$ with finite iteration of Cox rings and factorial master Cox ring. Let also $X=Y \git G$ be a quotient of a normal variety $Y$ by a solvable reductive group $G$, such that $Y$ has only constant invertible $G$-invariant functions and $G$ acts freely on a subset with complement of codimension at least two in $Y$.
Then the quotient $(X^{(N)})^{\mathrm{ss}} \to X$, where $(X^{(N)})^{\mathrm{ss}}$ denotes the set of semistable points in $X^{(N)}$, factors through $Y \to X$. We call $(X^{(N)})^{\mathrm{ss}} \to X$ the \emph{universal solvable quotient presentation} of $X$.
\end{introcor}

Note that this does not mean in general that a solvable quotient $Y \to X$ factors through $\widehat{X} \to X$, compare~\cite[\S 3]{ArGa} for the affine case. Moreover,  if $X$ has non-$\MDS$ master Cox ring, difficulties as with Theorem~\ref{thm:webCRI} arise.

\subsection*{Acknowledgements}
The author wishes to thank Giuliano Gagliardi, Helmut Hamm, Mitsuyasu Hashimoto, J\'anos Koll\'ar, Scott Nollet, Ivo Radloff and Karl Schwede for helpful comments. In particular, he is grateful to J\"urgen Hausen and Stefan Kebekus for stimulating discussions.

\setcounter{tocdepth}{1}
\tableofcontents

\section{Preliminaries}
\label{sec:prelim}

In this section, we shortly discuss notions and definitions relevant for the article, where for a detailed discussion we refer to~\cite{HuKeel, BerchHaus, coxrings, McKer}. Further definitions are given directly in the respective sections.

Notions for singularities of pairs are standard, we refer to~\cite{kollarmori}. A variety of \emph{Fano type} is a projective variety $X$, such that there exists an effective $\QQ$-divisor $\Delta$ with $(X,\Delta)$ $\klt$ and $-(K_X+\Delta)$ ample.

We use the following notions concerning the Picard and divisor class group: if $\Cl(X)$ is trivial (torsion), we say $X$ is factorial (almost factorial). If $\Cl(X)/\Pic(X)$ is trivial (torsion), we say that $X$ is locally factorial ($\QQ$-factorial). 

Our object of interest are Mori Dream Spaces ($\MDS$), which for us will be normal algebraic varieties with a finitely generated divisor class group and finitely generated Cox ring
$$
\Cox(X):=\bigoplus_{\Cl(X)} \Gamma(X,\mathcal{O}_X(D)).
$$
The original definition~\cite[Def.~1.10]{HuKeel} of Hu and Keel was in terms of a good behaviour with respect to the minimal model program, but they showed in~\cite[Prop.~2.9]{HuKeel} that this is equivalent to a finitely generated Picard group $\Pic(X)$ and $\Pic(X)$-graded Cox ring at least for $\QQ$-factorial projective varieties.

These definitions have been generalized to not necessarily separated prevarieties and $\Cl(X)$-graded Cox rings~\cite{CRC2, coxrings}. Sometimes only the free part of $\Cl(X)$ is taken, but the differences are then limited to finite abelian covers, see~\cite[Ch.~2,~5]{gongyo}. Finite generation of $\Pic(X)$ and $\Cl(X)$ is equivalent for $\QQ$-factorial varieties and in the non-$\QQ$-factorial case at least for $\klt$ pairs~\cite[Cor.~1.4.3]{birketal}. Non-$\QQ$-factorial $\MDS$ always have a \emph{small} $\QQ$-factorial birational modification~\cite[Thm.~2.3]{AHL}.

For technical reasons (at least in the case when $\Cl(X)$ has torsion), we always assume $\Gamma(X,\mathcal{O}^*)=\CC^*$, see~\cite[Sec.~1.4.2]{coxrings}. Note that for complete varieties and quasicones, this condition is satisfied \emph{and} preserved by the Cox ring, which is not always the case, see~\cite[Ex.~1.4.4.2]{coxrings}.

\section{Almost principal quasitorus bundles}
\label{sec:quotpres}

In~\cite{hashlongpaper}, Hashimoto introduced the notion of \emph{almost principal fiber bundles}.
The definition is as follows:

\begin{definition}
\label{def:hash}
Let $Y$ and $X$ be normal varieties, moreover let $G$ be an affine algebraic group, acting on $Y$ such that  $\varphi \colon Y \to X=: Y \git G$ is a good quotient. Then we call $\varphi$ an almost principal $G$-bundle if there exist open subsets $U \subseteq Y$ and $V \subseteq X$ with complement of codimension at least two so that
$$
\left. \varphi \right|_{U} \colon U \to V
$$ 
is a principal $G$-bundle. 
\end{definition}

This notion naturally comes into play in the setting of Cox rings, as the representation of an $\MDS$ $X$ as the quotient $\widehat{X}\git H_X$ is of such kind, see~\cite[Prop.~1.6.1.6]{coxrings}. In fact, $V$ can always be chosen to be the regular locus of $X$. Moreover, the quotient presentations from~\cite[Def.~4.2.1.1]{coxrings} having the property that the Cox construction factors through them are \emph{almost principal quasitorus bundles} with the additional assumption that all invertible functions homogeneous with respect to the grading group are constant. We will in the following use the notions quotient presentation and  almost principal quasitorus bundle interchangeably, where we assume only constant invertible homogeneous functions and that $V$ from Definition~\ref{def:hash} can be chosen as $X_{\rm reg}$.

Okawa has shown in~\cite{okawa} that if $f \colon X \to Y$ is a surjective morphism of projective varieties and $X$ is an $\MDS$, then $Y$ is as well. B\"aker showed in~\cite{bäk} that the same holds if $f \colon X \to Y=X \git G$ is a quotient presentation with not necessarily projective normal $X$ and a reductive affine algebraic group $G$. We can in fact say more if $G$ is a quasitorus:

\begin{proposition}
\label{prop:relMDS}
Let $\varphi \colon X \to Y$ be an almost principal $H$-bundle of normal varieties with $H$ a quasitorus. We have the following:
\begin{enumerate}
\item If $H$ is a torus, then $X$ is an $\MDS$ if and only if $Y$ is so and in this case $\Cox(X)=\Cox(Y)$.
\item If $X$ is an $\MDS$, then $Y$ is so and we have a commutative diagram of almost principal quasitorus bundles:
$$
 
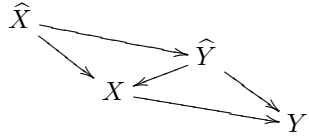
\xymatrix@C=2em@R=-0.3em@H=\dhatheightt{ 
 {\widehat{X}}
\ar[rrd]
\ar[rdd]
\\
& & 
{\widehat{Y}}
\ar[ld]
\ar[rdd]
\\
& 
X
\ar[rrd]
\\
&&&
Y
}
$$
\item If $Y$ is an $\MDS$, then its characteristic space is an almost principal quasitorus bundle over $X$, i.e. we have at least the following commutative subdiagram:
$$
 
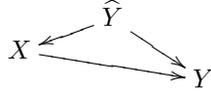
\xymatrix@C=2em@R=-0.3em@H=\dhatheightt{ 
 & 
{\widehat{Y}}
\ar[ld]
\ar[rdd]
\\
X
\ar[rrd]
\\
&&
Y
}
$$
Moreover, if $\widehat{Y}$ is an $\MDS$, then $X$ is so and we get the enhanced diagram 
$$
 
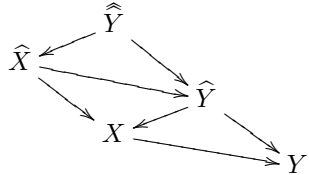
\xymatrix@C=2em@R=-0.3em@H=\dhatheightt{ 
 &
 {\doublewidehat{Y}}
 \ar[ld]
\ar[rdd]
 \\
 {\widehat{X}}
\ar[rrd]
\ar[rdd]
\\
& & 
{\widehat{Y}}
\ar[ld]
\ar[rdd]
\\
& 
X
\ar[rrd]
\\
&&&
Y
}
$$
\end{enumerate}
\end{proposition}

\begin{remark}
One can observe that here iteration of Cox rings comes into play, that is, if $Y$ has iteration of Cox rings with at least two steps, then $X$ is an $\MDS$. For further generalizations see Proposition~\ref{prop:relICR} and of course Theorem~\ref{thm:webCRI}.
\end{remark}

Before proving Proposition~\ref{prop:relMDS}, we state a corollary that is crucial for both Gorensteinness of Cox rings and finiteness of Cox ring iteration. It shows that among the quotient presentations of $X$, the two exceptional ones $\widehat{X}/E \to X$ and $\widehat{X} \git \TT \to X$ - with $E$ and $\TT$ being the torsion and torus part of $H_X$ - have very special properties. 

\begin{corollary}[Lemma~\ref{le:TandE}]
\label{cor:TandE}
Let $X$ be an $\MDS$ and $H_X= E \times \TT$ its characteristic quasitorus with torsion and torus part $E$ and $\TT$ respectively. We have the following:
\begin{enumerate}
\item The Cox rings of $X$ and of the geometric quotient $X_E:=\widehat X / E$ coincide. In particular, $X_E$ is a quasiaffine $\MDS$ and  almost factorial.
\item The characteristic space $\widehat{X}$ is an $\MDS$ if and only if $X_\TT:=\widehat X \git \TT$ is an $\MDS$ and in that case, their Cox rings coincide.
\end{enumerate}
\end{corollary}

\begin{proof}
The assertions follow directly if one of $E$ or $\TT$ is trivial - note that $\widehat{X}$ is factorial in the first case. So assume that none of them is trivial.

We begin with (i). By Proposition~\ref{prop:relMDS}, since $X_E \to X$ is an almost principal torus bundle, both share the same Cox ring. Thus $E$ is the characteristic quasitorus of $X_E$ and since it is finite and $X_E$ is open in $\overline{X}/E$, which is affine, it is quasiaffine and almost factorial. Also (ii) follows directly from (i) of Proposition~\ref{prop:relMDS}.
\end{proof}

\begin{remark}
In fact in Corollary~\ref{cor:TandE}~(i), we can replace $E$ by $E \times T$ for any subtorus $T \subseteq \TT$ and still $X_{E \times T}$ has Cox ring $\Cox(X)$, while quasiaffineness may fail. 
\end{remark}

\begin{proof}[Proof of Proposition~\ref{prop:relMDS}]
We consider the first assertion. 
If $X$ is an $\MDS$, then $\Cox(X)$ is factorially $\Cl(X)$-graded. Thus by~\cite[Thm.~1.5]{Bechtold}, it is also factorially $\Cl(X) \times \Chi(H)$-graded, since the character group $\Chi(H)$ of $H$ is torsion free. 
All involved actions are free over smooth loci. By~\cite[Thm.~1.6.4.3]{coxrings}, we get that $Y$ is an $\MDS$ and its Cox ring coincides with $\Cox(X)$. 

Now assume $Y$ is an $\MDS$. Then by~\cite[Thm.~4.2.1.4]{coxrings}, we have a commutative diagram of quotient presentations
$$
\xymatrix@C=7em{ 
{\widehat{Y}}
\ar[r]^{  \git (H_Y/H)  } 
\ar@/_1pc/[rr]_{ \git H_Y}
&
X 
\ar[r]^{  \git H } 
&
Y }.
$$ 
Since $\Cox(Y)$ is factorially $\Cl(Y)$-graded and $\Cl(Y) = \Chi(H_{Y}/H) \times \Chi(H)$, it is factorially $\Chi(H_{Y}/H)$-graded by~\cite[Thm.~1.5]{Bechtold}. By~\cite[Thm.~1.6.4.3]{coxrings}, the Cox ring of $X$ is $\Cox(Y)$.

We come to assertion (ii). If $X$ is an $\MDS$, then $Y$ is so by~\cite[Thm.~1.1]{bäk} and by~\cite[Thm.~4.2.1.4]{coxrings}, we get a quotient presentation $\widehat{Y} \to X$. Applying~\cite[Thm.~4.2.1.4]{coxrings} again, we get a quotient presentation $\widehat{X} \to \widehat{Y}$ and thus the desired diagram.
For (iii), apply~\cite[Thm.~4.2.1.4]{coxrings} several times as above.  
\end{proof}

The following proposition uses reduction to positive characteristic. We refer to~\cite{HoHu, HaWa, schwedesmith} for definitions and details, with an emphasis on Cox ring related problems in~\cite{gongyo, achilten}.
 
The main purpose of~\cite{hashlongpaper} is to show how properties are preserved by almost principal fiber bundles.
In the same spirit, varying~\cite[Prop.~5.20]{kollarmori}, we show the following on preservation of log-terminality.

\begin{proposition}
\label{prop:bundlelogterm}
Let $\varphi \colon X \to Y:= X \git G$ be an almost principal $G$-bundle with $G$ an affine algebraic group. Let $X$ and $Y$ be $\QQ$-Gorenstein. Then $X$ is log-terminal if and only if $Y$ is so.
\end{proposition}

\begin{proof}
By~\cite[Cor.~3.4]{takagi}, a pair $(Z,\Delta_Z)$ is Kawamata log-terminal if and only if it is of locally $F$-regular type in the sense that all local rings are of strongly $F$-regular type.
For reduction modulo $p$, see~\cite[Ch.~2]{gongyo} and~\cite[Ch.~2]{HoHu}. So the assertion follows from~\cite[Thm.~13.14]{hashlongpaper}.
\end{proof}

\begin{remark}
If~\cite[Thm.~13.14]{hashlongpaper} is valid for \emph{pairs} as well, then we can state Proposition~\ref{prop:bundlelogterm} also for pairs $(X,\Delta_X)$, $(Y,\Delta_Y)$.
Also note that there is a different definition of being of locally $F$-regular type than the one used above. Namely the a priori stronger one that $Z$ is covered by affine globally $F$-regular subsets, see~\cite[Def.~3.1]{schwedesmith}.  One can show that these two notions are in fact equivalent, see~\cite{schwede} with slightly different definitions.
\end{remark}

The statement of Proposition~\ref{prop:bundlelogterm} has a different, more local nature than the corresponding ones~\cite[Prop.~4.6]{gongyo} and~\cite[Prop.~7.17]{achilten}. These guarantee not only log-terminality but the global property of being of Fano type for a variety $X$ if the Cox ring is log-terminal. Combining these two viewpoints, we get the following.

\begin{corollary}
Let $X$ be a log-terminal $\MDS$ that is not of Fano type. Then the complement of $\widehat{X}$ in $\overline{X}$ contains all non log-terminal singularities of $\widehat{X}$.
\end{corollary}

\section{Cox rings are Gorenstein}
\label{sec:Gor}

The purpose of this section is to prove Theorem~\ref{thm:Gor}. For convenience, we state it here again in the following form:

\begin{theorem}[Theorem~\ref{thm:Gor}] 
\label{thm:gorenstein}
Let $X$ be an $\MDS$. Then $\Cox(X)$ is a numerically Gorenstein ring. In particular, if $\overline{X}$ is Cohen Macaulay,  it is Gorenstein. If $\overline{X}$ is rational, it is Gorenstein canonical.
\end{theorem}

\begin{remark}
In general, one cannot expect a finitely generated Cox ring to be Cohen Macaulay or $\QQ$-factorial. For a non Cohen Macaulay example by Gongyo, see the paper~\cite{brown} by Brown. For a collection of non $\QQ$-factorial examples, see~\cite[Thm.~1.9]{ltpticr}, where the toric one given by $xy+z^aw^b$ is probably the easiest one. It is the Cox ring of the affine threefold with $(\CC^*)^2$-action given by
$$
x^2+y^2z+z^aw^b.
$$
\end{remark}

\begin{proof}[Proof of Theorem~\ref{thm:gorenstein}]
If $H_X$ is torsion free, then $\overline{X}$ is factorial and thus numerically Gorenstein.
Otherwise, consider $Y:=\widehat{X}/E$, where $E$ is the torsion part of $H_X$. Then $Y$ is almost factorial by Corollary~\ref{cor:TandE}. In fact, every Weil divisor is $\QQ$-principal. Then a multiple of $K_X$ is principal, say $rK_X =0$ in $\Cl(Y)$. This means we have its \textit{global} canonical or index-$1$-cover $\rho \colon \mathcal{Y} \to Y$, see~\cite[5.19]{kollarmori}. Since $\rho$ is \'etale over $Y_\reg$, $\ZZ/r\ZZ$ acts strongly stable on $\mathcal{Y}$ so that $\rho$ becomes a quotient presentation of $Y$. Now~\cite[Thm.~4.2.1.4]{coxrings} gives us the commutative diagram of quotient presentations
$$
\xymatrix@C=7em{ 
{\widehat{X}}
\ar[r]^{  \git (H_{Y}/(\ZZ/r\ZZ))  } 
\ar@/_1pc/[rr]_{ \git H_{Y}}
&
{\mathcal{Y}} 
\ar[r]^{  \git (\ZZ/r\ZZ )} 
&
Y }.
$$ 
Note that even if $\Cox(\mathcal{Y})$ exists, it may differ from $\Cox(X)$. Now since $K_\mathcal{Y}$ is Cartier and $\widehat{X} \to \mathcal{Y}$ is unramified over $\mathcal{Y}_{\reg}$, also $K_{\widehat{X}}$ is Cartier, see~\cite[Prop.~5.20]{kollarmori}. So $\Cox(X)$ is numerically Gorenstein. The additions follow directly.
\end{proof}

In general, one cannot expect $K_{\widehat{X}}$ to be trivial, though $K_\mathcal{Y}$ is. An important class of varieties where this is true are the quasicones from Section~\ref{sec:quasicones}.


\section{Quasicone Cox rings}
\label{sec:quasicones}

In~\cite{birketal}, it was shown that varieties of Fano type are $\MDS$. So in a way, the Fano type property \emph{guarantees} $\MDS$-ness, while non Fano type varieties may or may not be $\MDS$. In this section, we give a criterion (namely being a $klt$ quasicone) for affine varieties that guarantees $\MDS$-ness and, moreover, guarantees \emph{preservation} of $\MDS$-ness by taking Cox rings.
This is exactly what we need for iteration of Cox rings. Since we also show that Cox rings of Fano varieties are such quasicones, one can say that also the Fano type property not only guarantees $\MDS$-ness but also its preservation. 

\begin{definition}
An affine variety $X=\Spec(A)$ is called a \emph{quasicone}, if one of the following equivalent conditions holds:
\begin{enumerate}
\item $X$ allows a $\CC^*$-action and the closures of all $\CC^*$-orbits meet in one common point.
\item $A$ is $\ZZ_{\geq 0}$ graded and $A_0=\CC$.
\item $A$ has homogeneous generators $a_1,\ldots,a_n$ of strictly positive degree. 
\end{enumerate}
\end{definition}

For a set of homogeneous generators of $A$ as in (iii) and the corresponding embedding of $X$ into $\CC^n$, the origin is the common point of the orbit closures from (i), it is called the \emph{vertex} of $X$. We refer to~\cite{Dem, Dolga} for detailed treatments of quasicones. 

\begin{lemma}
Let $X$ be a normal quasicone. Then $\Pic(X)=0$.
\end{lemma}

\begin{proof}
This is~\cite[Le.~5.1]{murthy}, compare also~\cite[Cor.~10.3]{Fossum} and~\cite[Cor.~2.14]{ltpticr}.
\end{proof}

\begin{remark}
\label{rem:PicNorm}
If $X$ is not normal, $\Pic(X)$ can be nontrivial. Consider for example $A=\CC[x,y]/(x^3-y^2)$ with $\Pic(A)=\CC_+$.
\end{remark}

\begin{proposition}
\label{prop:quasicone2}
Let $X$ be an $\MDS$ that is either a quasicone or has only constant global regular functions, e.g. $X$ is complete. Then $\overline{X}$ is a \emph{quasicone}. If $x$ is the vertex of $\overline{X}$, $\Cl(\overline{X},x) \cong \Cl(\overline{X})$ and every irreducible component of $\Sing(\overline{X})$ intersects $x$ nontrivially.
\end{proposition}

\begin{proof}
First assume $X$ has only constant global regular functions.
We have a decomposition $H_X=E \times \TT$ of the characteristic quasitorus of $X$ with nontrivial torus part $\TT$, since otherwise $X=X_E$ would be quasiaffine with Proposition~\ref{cor:TandE}. Then the finite morphism $X_\TT \to X$ is proper, hence $X_\TT$ is complete. Consider the $\ZZ^k$-grading of $A:=\Cox(X)$ corresponding to the quotient $\widehat X \to X_\TT$, observe $A_0=\CC$ and set $\mathfrak{m}=A \setminus A_0$. Since the weight cone contains no line, by appropriately projecting onto $\ZZ$, we can assume $k=1$ and $\overline{X}$ is a quasicone. 

Now let $X$ be a quasicone. Here the coordinate ring $A:=\CC[X]$ is $\ZZ_{\geq 0}$-graded and generators \emph{and relations} of $A$ are of strictly positive degree. The grading of $A$ can be lifted to $\Cox(X)$ and thus $\Cox(X)$ is $\ZZ^{\dim(H)+1}$ graded, where $H$ is the characteristic quasitorus of $X$. As we have $\Cox(X)_{0_H}=A_0=\CC$, the $\ZZ^{\dim(H)+1}$-weight cone of $\Cox(X)$ again contains no lines and by appropriately projecting onto $\ZZ$, we see that $\overline{X}$ is a quasicone.

The remaining assertions follow directly from $X$ being a quasicone.   
\end{proof}

\begin{remark}
If an affine $\MDS$ $X$ is a quasicone, the fan of the canonical ambient toric variety $Z$, see~\cite[Sec.~3.2.5]{coxrings}, is in fact (the fan of faces of) \emph{a cone}, since the torus action of $X$ comes from the torus action on $Z$ and thus the vertex of $X$ is the torus fixed point of $Z$. 
In general, the fan of the canonical ambient toric variety of affine $\MDS$ that are not quasicones is only a subfan of the fan of faces of a cone. 
\end{remark}

\begin{proof}[Proof of Theorem~\ref{thm:charsing}]
If $X$ is projective, then the assertion follows from~\cite[Thm.~1.1]{gongyo} and Theorem~\ref{thm:gorenstein}, Proposition~\ref{prop:quasicone2}.
If $X$ is affine, then let $(X,\Delta)$ be $\klt$. In particular, since $\Pic(X)$ is trivial, an integer multiple of  $K_X+\Delta$ is trivial. By~\cite[Cor.~1.4.3]{birketal} with $\mathfrak{E}=\emptyset$, there is a \emph{small} birational contraction $\varphi \colon Y \to X$, which is a log terminal model. There is a $\QQ$-divisor $\Delta'\geq 0$ on $Y$ such that $K_Y+\Delta'=\varphi^*(K_x+\Delta)$. It may occur that $K_Y+\Delta'$ has no trivial $\NN$-multiple any more, but it is still $\QQ$-factorial and nef. 
By~\cite[Cor.~3.9.2]{birketal}, $K_Y+\Delta'$ is semiample and thus by the proof of~\cite[Cor.~1.1.9]{birketal}, see also~\cite[Sec.~1.3]{birketal}, $\Cox(Y)$ and thus $\Cox(X)$ is finitely generated. Since $X$ is $\klt$, it is of globally or equivalently locally $F$-regular type, see~\cite[Cor.~3.4]{takagi} and for the equivalence~\cite{schwedesmith}. On the other hand, since $\Cox(X)$ is a Gorenstein quasicone by Theorem~\ref{thm:gorenstein} and Proposition~\ref{prop:quasicone2}, it is of $F$-regular type if and only if it is canonical. Now we follow exactly the lines of~\cite[Proof of Thm.~4.7]{gongyo} and see that $\Cox(X)$ is canonical. 
If $X$ is affine and $\Cox(X)$ is a Gorenstein canonical quasicone, we still have to show that $X$ is a quasicone, i.e. the converse of Proposition~\ref{prop:quasicone2} in the affine case. But this is clear since if $T_X \subseteq \TT$ is the torus part of $H_X$ inside the maximal torus $\TT$ acting on $\overline{X}$, then with respect to the corresponding $\ZZ^{\dim(T_X)}$-grading of $\Cox(X)$, we have $\CC \neq \CC[X]\subseteq \Cox(X)_0$, i.e. $\dim(\TT /T_X)\geq 0$ and $\CC[X]_0=\CC$ with respect to the corresponding $\ZZ^{\dim(\TT /T_X)}$-grading. By projecting to $\ZZ$ as in the proof of Proposition~\ref{prop:quasicone2}, we see that $X$ is a quasicone.
\end{proof}

\section{Iteration of Cox rings}
\label{sec:ICR}
Iteration of Cox rings has been introduced for varieties with a torus action of complexity one in the work~\cite{ltpticr} by Arzhantsev, Hausen, Wrobel and the author. 

\begin{definition}
Let $X$ be an $\MDS$. If the characteristic space $X^{(1)}:=\widehat{X}$ or equivalently the total coordinate space $\mathcal{X}^{(1)}:=\overline{X}$ is an $\MDS$ as well, we can consider their characteristic spaces $X^{(2)}:=\doublewidehatX{X}$ and $\mathcal{X}^{(2)}:=\doubleoverline{X}$. By iterating this procedure, we get a commutative diagram
\begin{equation}
\label{eq:CRI}
\begin{gathered}
\xymatrix{ 
\ar@{..>}[r]
&
\mathcal{X}^{(3)}
\ar[r]^{\git H_2}
&
\mathcal{X}^{(2)}
\ar[r]^{\git H_1}
&
\mathcal{X}^{(1)}
\\
\ar@{..>}[r]
&
X^{(3)}
\ar[r]^{\git H_2}
\ar[u]
&
X^{(2)}
\ar[r]^{\git H_1}
\ar[u]
&
X^{(1)}
\ar[r]^{\git H_0}
\ar[u]
&
X
}
\end{gathered}
\end{equation}
of $1$-Gorenstein varieties (with the possible exception of $X$).
Here horizontal arrows stand for quotients by characteristic quasitori and vertical arrows stand for the inclusions of the $X^{(i)}$ as open subsets with complement of codimension at least two in $\mathcal{X}^{(i)}$. We call this the \emph{iteration of Cox rings} of $X$.  If $X$ is affine, then so are the $X^{(i)}=\mathcal{X}^{(i)}$.
\end{definition}

\begin{remark}
In the \emph{iteration of Cox rings} of an $\MDS$ $X$, we have three possibilities:
\begin{itemize}
\item For each $i \in \NN$, $X^{(i)}$ is an $\MDS$ with nontrivial divisor class group $\Cl(X^{(i)})$. We say that $X$ has \emph{infinite iteration of Cox rings} and set $N:=\infty$.
\item For some $N \in \NN$, either $X^{(N)}$ is not an $\MDS$ or it is factorial.
We call $\mathfrak{R}(X):=\KK[X^{(N)}]$ the \emph{master Cox ring} of $X$. In this case the diagram~\ref{eq:CRI} is finite and we say that $X$ has \emph{finite iteration of Cox rings}.
\begin{itemize}
\item If $X^{(N)}$ is factorial, we say that $X$ has factorial master Cox ring.
\item If $X^{(N)}$ is not an $\MDS$, we say that $X$ has non-$\MDS$ master Cox ring.
\end{itemize}
\end{itemize}
\end{remark}

\begin{proposition}
\label{prop:facquot}
For each natural number $i\leq N$, we can represent $X$ as a quotient of $X^{(i)}$ by a solvable group $G_i$ acting freely on a subset with complement of codimension at least two and the chain of abelian quotients
$$
 \xymatrix{ 
X^{(i)}
\ar[r]^{\git H_{i-1}}
&
\cdots
\ar[r]^{\git H_2}
&
X^{(2)}
\ar[r]^{\git H_1}
&
X^{(1)}
\ar[r]^{\git H_0}
&
X
}
$$
 can be retrieved by setting $H_j:=G_i^{(j-1)}/G_i^{(j)}$ and $X^{(j)}:=X^{(i)} \git G_i^{(j-1)}$ for the $k$-th derived subgroups $G_i^{(k)}$ of $G_i$.
\end{proposition}

\begin{proof}
By~\cite[Proof of Thm.~1.6]{ltpticr}, for each $i \leq N$, we can represent $X$ as a quotient of $X^{(i)}$ by solvable groups $G_i$ defined by $G_1:=H_0$ and a certain semidirect product $G_k:=H_{k-2} \rtimes G_{k-1}$ for $k\geq 2$. Now fix some $i \leq N$ and set $\mathcal{D}_k:=G_i/G_{k}$. Then $X^{(j)}=X^{(i)} \git \mathcal{D}_{j}$.
We can follow~\cite[Proof of Thm.~1.6]{ltpticr}, where since we do not assume $X^{(i)}$ to be factorial, we have to check that each $\mathcal{D}_k$-stable divisor on $X^{(i)}$ is principal in order to apply~\cite[Prop.~3.5]{ArGa}. But this holds since $X^{(i)} \to X^{(i)} \git \mathcal{D}_k$ factors through the characteristic spaces 
$$
X^{(i)} \to X^{(i-1)} \to \ldots \to X^{(k)}=X^{(i)} \git \mathcal{D}_{k}.
$$
Thus we have that $[\mathcal{D}_k,\mathcal{D}_k]=\mathcal{D}_{k+1}$ and the assertion follows.
\end{proof}

If $X$ has finite iteration of Cox rings, we call $G_X:=G_N$ the \emph{characteristic solvable group} of $X$.

\begin{example}
It is clear that toric varieties trivially have finite iteration of Cox rings with polynomial master Cox ring. On the other hand, in~\cite{gagliardi}, it was shown that \emph{spherical} varieties have finite iteration of Cox rings with factorial master Cox ring and at most two iteration steps.

In~\cite{ltpticr}, Arzhantsev, Hausen, Wrobel and the author showed that for log terminal singularities with a torus action of complexity one, Cox ring iteration is finite with factorial master Cox ring. Moreover, the possible chains for these singularities have been calculated explicitly in~\cite[Rem.~6.7]{ltpticr}. In dimension two, one retrieves exactly the representation of the log-terminal singularities as finite quotients of $\CC^2$, i.e. the Cox ring iteration chains of those singularities form a tree with the single root $\CC^2$, see~\cite[Ex.~4.8]{ltpticr}. In~\cite[Thm.~1.9]{ltpticr}, also the Cox ring iteration tree for \emph{compound du Val} threefold singularities with a two-torus action has been given, and it can be seen that the Cox rings preserve the compound du Val property. 
Moreover, in the work~\cite{canonical} by H\"attig and the author one can find the Cox ring iteration tree of \emph{canonical} threefold singularities with a two-torus action. It turns out that all master Cox rings here are compound du Val as well. 

In~\cite{hauswrob}, Hausen and Wrobel gave criteria for affine varieties with a torus action of complexity one to have finite iteration of Cox rings with factorial master Cox ring. In fact, one can directly verify from their computations that all varieties with a torus action of complexity one have finite iteration of Cox rings, though with possibly non-$\MDS$ master Cox ring.
\end{example}

The following \emph{relative} version of Cox ring iteration is the first generalization of Proposition~\ref{prop:relMDS}, from which then directly follows the most general version, the \emph{web of Cox ring iterations} from Theorem~\ref{thm:webCRI}.

\begin{proposition}[Relative iteration of Cox rings]
\label{prop:relICR}
Let the quasitorus $H$ act on $X$, such that the good quotient $Y:=X\git H$ exists and is an almost principal $H$-bundle. Then $X$ has finite iteration of Cox rings if and only if $Y$ has so. Moreover, $X$ has factorial master Cox ring if and only if $Y$ has so and in this case, $\mathfrak{R}(X)=\mathfrak{R}(Y)$.
\end{proposition}

\begin{proof}
If $H$ is torsion free, by Proposition~\ref{prop:relMDS} (i), $X$ and $Y$ have the same Cox ring, i.e. the same iteration of Cox rings from the second step on.
So assume $H$ has torsion. From Proposition~\ref{prop:relMDS} (iii), we get a Cox ring iteration ladder:
$$
 
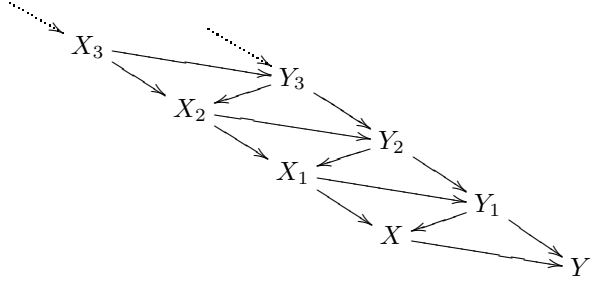
\xymatrix@C=2em@R=-0.3em@H=\dhatheightt@W=\Ywidth{ 
 \ar@{..>}[rdd]
 \\
   &&
\ar@{..>}[rdd]
 \\
& X_3
 \ar[rrd]
\ar[rdd]
\\
&&  &
 Y_3
 \ar[ld]
\ar[rdd]
 \\
&& X_2
 \ar[rrd]
\ar[rdd]
\\
&&& &
 Y_2
 \ar[ld]
\ar[rdd]
 \\
&&& X_1
\ar[rrd]
\ar[rdd]
\\
&&&& & 
Y_1
\ar[ld]
\ar[rdd]
\\
&&&& 
X
\ar[rrd]
\\
&&&&&&
Y
}
$$
So if $X_i$ is an $\MDS$, then so is $Y_i$ and if $Y_{j+1}$ is an $\MDS$, then so is $X_j$. Thus $X$ has finite iteration of Cox rings if and only if $Y$ has so. 
If $\mathfrak{R}(X)$ is the factorial master Cox ring of $X$, then $X_N$ is an almost principal quasitorus bundle over $Y_{N}$ and $\mathfrak{R}(X)$ is factorially graded with respect to the associated character group, since it is factorial. So $\mathfrak{R}(X)$ is the Cox ring of $Y_N$ and thus the master Cox ring of $Y$. If on the other hand $\mathfrak{R}(Y)$ is the factorial master Cox ring of $Y$, then $Y_N$ is an almost principal quasitorus bundle over $X_{N-1}$ and we can apply the same argument.
\end{proof}

\begin{proof}[Proof of Corollary~\ref{cor:unisolv}]
Let $Y \to X$ be a quotient by the solvable group $G$ as in the corollary. Then by any normal series of $G$ we get a chain of quotient presentations
$
Y \! \to \cdots\!\to \!  X_3 \! \to \! X_2 \! \to \! X.
$
Thus by Theorem~\ref{thm:gorenstein}, $Y$ is an $\MDS$ with factorial master Cox ring $\mathfrak{R}(X)$ and the diagram 
$$
\xymatrix@C=5em{ 
{X^{(N)}}
\ar[r]^{  \git G_Y  } 
\ar@/_1pc/[rr]_{ \git G_X}
&
{Y} 
\ar[r]^{  \git G} 
&
X }.
$$ 
is commutative, which proves the assertion.
\end{proof}

Now we come to the proof of our last main theorem.

\begin{theorem}[Theorem~\ref{thm:CRI}]
\label{thm:ICR}
Let $X$ be of Fano type or a $\klt$ quasicone. Then $X$ has finite iteration of Cox rings with factorial master Cox ring.
\end{theorem}

\begin{proof}
Let $X$ be a $\klt$ quasicone and assume it has infinite iteration of Cox rings. If $H_{i}$ is a torus for some $i \in \NN$, then $X_{i+1}$ is factorial by~\cite[Prop.~1.4.1.5]{coxrings}. So each $H_i$ has a nontrivial torsion part $E_i$ and a possibly trivial torus part $\TT_i$.

Denote $\Xendl_1:=X_{\TT}=X^{(1)} \git \TT_0$. By Corollary~\ref{cor:TandE}, $X^{(2)}$ is the characteristic space of $\Xendl_1$. Denote by $\mathfrak{H}_1$ the characteristic quasitorus of $\Xendl_1$ and by $\mathfrak{T}_1$ its torus part. Define $\Xendl_2:=X^{(2)} \git\mathfrak{T}_1$. By iteration of Cox rings, we get the following infinite commutative diagram, where all diagonal arrows are characteristic spaces:
$$
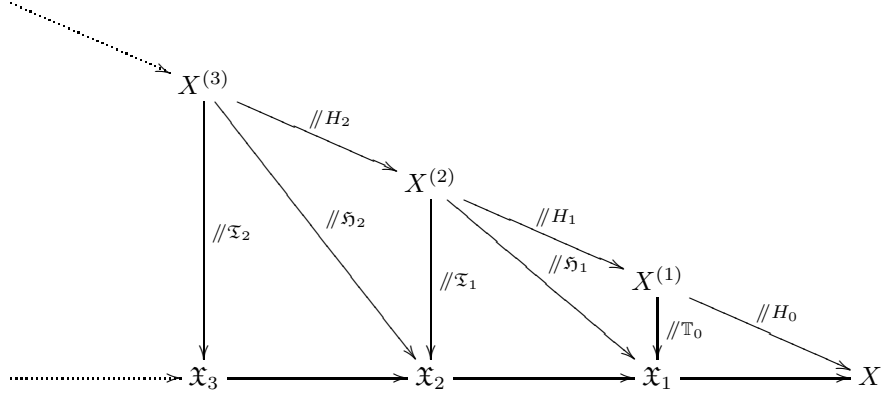
\xymatrix@C=6em@R=2em@H=\dhatheightt@W=\Ywidth{ 
 \ar@{..>}[rd]
 \\
   &
 X^{(3)}
\ar[rd]^{\git H_2}
\ar[ddd]^{\git \mathfrak{T}_2}
\ar[dddr]^{\git \mathfrak{H}_2}
\\
&&
 X^{(2)}
\ar[rd]^{\git H_1}
\ar[dd]^{\git \mathfrak{T}_1}
\ar[ddr]^{\git \mathfrak{H}_1}
\\
&&&
 X^{(1)}
\ar[rd]^{\git H_0}
\ar[d]^{\git \TT_0}
\\
 \ar@{..>}[r]
&
\Xendl_3
\ar[r]
&
\Xendl_2
\ar[r]
&
\Xendl_1
\ar[r]
&
 X
}
$$  
The bottom of this diagram gives a sequence of finite Galois morphisms between $\klt$ quasicones \'etale in codimension one such that all compositions are Galois as well, compare the lifting of the abelian group actions from Proposition~\ref{prop:facquot}. By~\cite[Thm.~1.1]{grebkebpet}, all but finitely many of the  $\Xendl_{i+1} \to \Xendl_i$ must be \'etale. But none of them is, since they are all ramified over the vertex, see~\cite[Prop.~1.6.2.5]{coxrings}. So $X$ must have finite iteration of Cox rings with master Cox ring $\mathfrak{R}(X)=\CC[X^{(N)}]$ for some $N \in \NN$.

Since all $X^{(i)}$ in the iteration of Cox rings of $X$ are Gorenstein canonical quasicones and thus $\MDS$ by Theorem~\ref{thm:charsing}, $X^{(N)}$ is an $\MDS$ and thus factorial. The assertion for varieties of Fano type follows directly.
\end{proof}

\section*{Questions} 
We conclude with a number of questions arising in the context of our observations.
The first one concerns the three possibilities of Cox ring iteration:

\begin{question}
Does every $\MDS$ have \emph{finite} iteration of Cox rings?
\end{question}

As we mentioned, this is true for $\MDS$ with torus action of complexity one. So we tend to say yes. But let us take a closer look at the circumstances. In principle, an affine (and thus from the viewpoint of Cox ring iteration any) $\MDS$ $X$ can differ in two ways from $\klt$ quasicones. 

On the one hand, $X$ can have nontrivial Picard group. This can lead to non-$\MDS$ master Cox rings even if $X$ is log terminal, see~\cite{hauswrob}. If one is only interested in local properties, by taking section rings for $\Cl(X)/\Pic(X)$ one should arrive at \emph{locally factorial} master section rings at least if $X$ is $\klt$. If we stick to Cox rings, then if some $X^{(i)}$ in the iteration of Cox rings becomes locally factorial, $X^{(i+1)} \to X^{(i)}$ is a free quasitorus bundle, so we cannot apply~\cite[Thm.~1.1]{grebkebpet} any more. The point is that such $X^{(i)}$ is \emph{$A_2$-maximal}, so it can not be embedded as an open subvariety with complement of codimension at least two in an \emph{$A_2$-variety} (i.e. a variety such that any two points admit a common affine open neighbourhood). 
But the freeness on the other hand leads us to the next question, where an affirmative answer would guarantee finiteness of Cox ring iteration for affine $\klt$ $\MDS$:

\begin{question}
Let $X$ be a locally factorial $A_2$-maximal $\MDS$. Then is it true that $\overline{X}$ is either factorial or non-$\MDS$?
\end{question}

We can of course reduce this question to the case when $\Pic(X)$ has torsion, since if not we already know that $\overline{X}$ is factorial.

The second way an affine $\MDS$ $X$ can differ from $\klt$ quasicones is having bad singularities. Let us isolate this case from the other one by assuming $X$ is a quasicone. Then by our reduction to finite (solvable) covers, this case becomes closely related to questions about local fundamental groups of non log terminal singularities. Even \emph{\'etale} local fundamental groups of non log terminal singularities can be infinite already in the surface case: triangle singularities are $\MDS$ quasicones with finite iteration of Cox rings (with possibly non-$\MDS$ master Cox ring) but their local fundamental groups are the  triangle groups, which are residually finite, i.e. their profinite completions are infinite. We can still hope for the following in the spirit of~\cite[Thm.~1.1]{grebkebpet}, compare also~\cite[Thm.~1]{Xu},~\cite{lafacefundgroup} and~\cite[Question~26]{kollnew}:

\begin{question}
Let $X$ be a non log terminal quasicone and $\cdots \to X_2 \to X_1 \to X$  a sequence of finite \emph{abelian} covers \'etale over smooth loci. Then is it true that all but finitely many are \'etale?
\end{question} 

Note that this question naturally leads to the notion of a universal \emph{solvable} cover. 

We come to our last and most vague question. The quasicones can be seen as affine models of singularities, reflecting only local properties without being truly local, which makes it possible to study properties such as quotient presentations of the singularity with Cox ring methods. So the following question comes up:

\begin{question}
\label{q:4}
Is every normal singularity (algebraically? analytically?) isomorphic to the vertex of a quasicone? At least to a point of an affine variety with trivial Picard group?
\end{question}

On the one hand, analytic isomorphy looks promising due to local contractability of singularities, while local contractability tends to lead to trivial Picard group~\cite{hammtrang} or even to being isomorphic to a quasicone, at least in the case of hypersurfaces, see~\cite{saito}. But on the other hand, the local class group can then be changed by algebraization, see~\cite{ParVStra}, which we do not want. Our impression is  that in order to answer this question in an appropriate way, one first needs an appropriate type of isomorphy.

Note that at least for singularities sitting in affine $\MDS$ we can answer the second part of Question~\ref{q:4} in the affirmative: the singularity corresponds to a cone of the fan of the ambient toric variety, see~\cite[Sec.~3.3.1]{coxrings}, restricting to this cone means restricting to an affine neighbourhood of the singularity with trivial Picard group by~\cite[Cor.~3.3.1.7]{coxrings}.

\end{document}